\documentclass[12pt]{extarticle}
\usepackage{amsmath, amsthm, amssymb, color}
\usepackage[colorlinks=true,linkcolor=blue,urlcolor=blue]{hyperref}
\usepackage{graphicx}
\usepackage{caption}
\usepackage{mathtools}
\usepackage{enumerate}
\usepackage{verbatim}
\usepackage{tikz,tikz-cd,tikz-3dplot}
\usepackage{amssymb}
\usetikzlibrary{matrix}
\usetikzlibrary{arrows}
\usepackage{algorithm}
\usepackage[noend]{algpseudocode}
\usepackage{caption}
\usepackage[normalem]{ulem}
\usepackage{subcaption}
\oddsidemargin \evensidemargin
\usepackage{makecell}
\usepackage{array}

\newtheorem{theorem}{Theorem}
\newtheorem{proposition}[theorem]{Proposition}
\newtheorem{lemma}[theorem]{Lemma}

\theoremstyle{definition}

\newtheorem{remark}[theorem]{Remark}

\newtheorem{example}[theorem]{Example}

\newcommand{\PP}{\mathbb{P}}
\newcommand{\RR}{\mathbb{R}}
\newcommand{\QQ}{\mathbb{Q}}
\newcommand{\CC}{\mathbb{C} }
\newcommand{\ZZ}{\mathbb{Z}}
\newcommand{\NN}{\mathbb{N}}

\newcommand{\kk}{\kappa}

\renewcommand{\ll}{\lambda}

\title{\bf Toric Geometry of Entropic Regularization}

\author{ 
Bernd Sturmfels, Simon Telen, \\
Fran\c{c}ois-Xavier Vialard, and Max von~Renesse}

\date{}
\begin{document}
\maketitle

\begin{abstract} \noindent
Entropic regularization is a method for 
large-scale linear programming. Geometrically, one traces intersections of
the feasible polytope with scaled toric varieties, starting at the Birch point.
We compare this to log-barrier methods, with reciprocal linear~spaces, 
starting at the analytic center.
We revisit entropic regularization for unbalanced optimal transport, 
and we develop the use of optimal conic couplings.
We compute the degree of the associated toric variety, and we explore
algorithms like iterative scaling.  \end{abstract}

\section{Introduction}
\label{sec1}

Linear programming in standard form is the optimization problem
\begin{equation}
\label{eq:LP}
{\rm Minimize} \,\,\,c \cdot x \,\,\,\hbox{subject to}\, \,\,A x = b
\,\,{\rm and} \,\, x \geq 0 .
\end{equation}
Here $A$ is a nonnegative $d \times n$ matrix of rank $d$ with no zero column,
$c \in \RR^n$ is a row vector, and
$b \in \RR^d$ is a column vector.
This program is feasible if and only if $b$ lies in 
${\rm pos}(A)$, which is the convex polyhedral cone spanned by the columns of $A$.
If $c$ is fixed and generic, and $b$ ranges over ${\rm pos}(A)$, then the set of optimal bases of (\ref{eq:LP})
defines a regular triangulation of the cone ${\rm pos}(A)$. 
This  classical result due to Walkup and Wets is explained geometrically in \cite[Theorem 1.2.2]{DRS}.
The triangulation is replaced by a continuous shape under a {\em regularization}
\begin{equation}
\label{eq:entropicLP}
{\rm Minimize}\, \,\,c \cdot x \,+\, \epsilon \sum_{i=1}^n H(x_i) \,\,\,\,\hbox{subject to} \,\,\,A x = b
\,\,{\rm and} \,\, x \geq 0 .
\end{equation}
Here, $H $ is a strictly convex smooth function on $\mathbb R_{\geq 0}$, and $\epsilon$ is a  positive parameter. For interior point methods, $H$ is taken as a barrier function, meaning that its limit at $0$ is $+\infty$.
This enables us to remove the constraint $x\geq 0$ in \eqref{eq:entropicLP}.
The dual formulation of \eqref{eq:entropicLP} reads:
\begin{equation}\label{EqDualProblem}
\text{Maximize}\, \,\,\,  b \cdot p - \epsilon \sum_{i = 1}^n H^*\left(
\frac{1}{\epsilon}
\left[A^\top p - c \right]_i\right)\,\,\,
\hbox{over all} \,\,\, p \in \RR^d.
\end{equation}
Here,  $\,H^*(s) = \sup_{t \in \mathbb R}( st - H(t))\,$
denotes the {\em Legendre-Fenchel transform} of the convex function~$H$,
after the latter has been extended to all of $\mathbb R$ by setting $H(t) = +\infty$ for $t<0$.

The feasible set
$P_{A,b} := \{ x \in \RR^n_{\geq 0} : Ax = b \,\}$ for (\ref{eq:LP}) is a polytope. For every $\epsilon > 0 $, the 
regularized problem
 (\ref{eq:entropicLP}) has a unique optimal solution $x^*(\epsilon)$ in the
relative interior of $P_{A,b}$, provided the function $H$ is barrier.
The curve $\,\mathcal{C}_{A,b,c} = \{ x^*(\epsilon) \,: \, 0 \leq \epsilon \leq \infty\}$
connects the distinguished point $x^*(\infty)$ in $P_{A,b}$ to an
optimal solution $x^*(0)$ of the linear program (\ref{eq:LP}).

Applying Lagrange multipliers to \eqref{eq:entropicLP} gives a determinantal representation for $\mathcal{C}_{A,b,c}$:
\begin{equation}
\label{eq:curveconstraints}
 A x = b \qquad {\rm and} \qquad {\rm rank} \begin{pmatrix} A \\ c \\ H'(x) \end{pmatrix} \,\leq \, d+1 .
 \end{equation}
The matrix on the right has $d+2$ rows and $n$ columns.
 Its last row is the vector of derivatives
$$ H'(x) \,\,= \,\, \bigl(H'(x_1),H'(x_2),\ldots,H'(x_n) \bigr). $$
For generic cost vectors $c$, the number of independent constraints in (\ref{eq:curveconstraints}) equals
$d + (n-d-1) = n-1$, so we expect these to cut out an analytic curve in $\RR^n$.
The distinguished interior point  $x^*(\infty)$, at which our curve starts, satisfies
$ {\rm rank} \begin{pmatrix} A \\ H'(x) \end{pmatrix}  \leq d $.
For any fixed $\epsilon \in \RR_{> 0}$, the point $x^*(\epsilon)$ on the curve satisfies
$ {\rm rank} \begin{pmatrix} A \\ c + \epsilon H'(x) \end{pmatrix}  \leq d $. 
Moreover, if $H$ and $c$ satisfy certain hypotheses then the curve is algebraic, and we can 
study its defining ideal in $\RR[x_1,x_2,\ldots,x_n]$.

We compare two widely used regularizations.
The first is the {\em logarithmic barrier function} $H(t) = -{\rm log}(t)$,
where (\ref{eq:entropicLP}) is the standard formulation of an
 interior point method for (\ref{eq:LP}).
  This function is  self-concordant, which is a key property in convex optimization.
  The rank condition in (\ref{eq:curveconstraints}) translates into polynomials
by taking numerators of all maximal minors. These define an
algebraic curve $\mathcal{C}^{R,+}_{A,b,c}$ in the polytope $P_{A,b}$. This is known as the
{\em central path}. Its starting point
$x^*(\infty)$ is the {\em analytic center} of $P_{A,b}$.
The algebraic complexity of these objects are governed by
the bounded regions in certain hyperplane arrangements.
See \cite{ABGJ, DSV}.

Next consider the {\em entropy function} $H(t) = t \cdot {\rm log}(t) - t $,
whose Legendre-Fenchel transform equals $H^*(s) = {\rm exp}(s)$.
Here, (\ref{eq:entropicLP}) is the {\em entropic regularization} of (\ref{eq:LP}).
This approach is popular in machine learning, especially for 
 optimal transport problems \cite{Cuturi, KR, Weed}. 
 Note that $H(t)$ is strictly convex but not strongly convex. It is not a barrier function since 
 $H(t)$ does not diverge for $t \rightarrow 0$.
 But, its derivative does, and this ensures the minimizer $x^*(\epsilon)$ to be in the 
 relative interior of $P_{A,b}$.
  To highlight algebraic features, we assume
 that the cost vector $c$ has integer coordinates.
The rank condition in (\ref{eq:curveconstraints}) is a system of $\ZZ$-linear equations
 in ${\rm log}(x_1), \ldots, {\rm log}(x_n)$.
 These translate into
 differences of monomials in $\RR[x_1,\ldots,x_n]$.
Indeed, (\ref{eq:curveconstraints}) 
specifies the toric variety of the integer matrix $\binom{A}{c}$.
   We obtain the {\em entropic
 curve} $\mathcal{C}^T_{A,b,c}$ by
 intersecting that toric variety
 with the linear space $\{Ax = b\}$.
 Its degree is bounded by the normalized volume of a polytope
associated to $ \binom{A}{c} $, and 
 $x^*(\infty)$ is the {\em Birch point} of $P_{A,b}$.

\begin{example}[$d=4,n=6$] \label{ex:trans23}
We consider the transportation problem of format $2 \times 3$,
as in \cite[Examples 2 and 14]{DSV}. We  here represent this by a matrix
with linearly independent rows:
\begin{equation} \label{eq:A23}
A \,\, = \,\, 
\begin{pmatrix} 
1 & 1 & 1 & 0 & 0 & 0 \\
0 & 0 & 0 & 1 & 1 & 1 \\
1 & 0 & 0 & 1 & 0 & 0 \\
0 & 1 & 0 & 0 & 1 & 0 \\
\end{pmatrix}.
\end{equation}
To explore the generic behavior for this $A$, we fix
$\, b = (7,8,4,5)^\top \,$ and $\, c = (1,0,1,\,0,2,5)  $.
The transportation polytope $P_{A,b}$ is a hexagon in the affine plane
$\{Ax=b\}$ in $\RR^6$. 
We use coordinates $(x_1,x_2)$, as these determine $x_3,x_4,x_5,x_6$.
First consider its log-barrier geometry.
The edges of
$P_{A,b}$ specify an arrangement of lines in 
the plane $\{Ax = b \}$, whose complement has seven bounded regions. Therefore, the analytic center
has algebraic degree~seven:
$$ x^*(\infty) = ( 1.895889342, 2.337573614, 2.766537044 ,\, 2.104110658, 2.662426386, 3.233462956 ). $$
For generic cost vectors $c$, the central path has degree five.  Two pictures are shown in \cite[Figure 1]{DSV}.
For the specific $c$ above, the quintic polynomial defining the central path equals
$$ \begin{small}
10 x_1^4 x_2+22 x_1^3 x_2^2+8 x_1^2 x_2^3-4 x_1 x_2^4-25 x_1^4-180 x_1^3 x_2-183 x_1^2 x_2^2 + \cdots +376 x_2^2+700 x_1-280 x_2. \end{small}
$$

Now compare this to entropic regularization.
The Birch point has rational coordinates:
$$ x^*(\infty)\, =\,  \frac{1}{15}(28,35,42,\,32,40,48) \, = \,
(1.8666, 2.3333, 2.8000, \,2.1333, 2.6666, 3.2000). $$
The rank constraint in (\ref{eq:curveconstraints}) translates into the binomial equation
$\, x_1^2 x_3^3 x_5^5 =  x_2^5 x_4^2 x_6^3 $.
      The degree drops by one
when we intersect with $\{Ax=b\}$.
The entropic curve is given by
$$ \begin{small} 25 x_1^5 x_2^4+85 x_1^4 x_2^5+87 x_1^3 x_2^6+19 x_1^2 x_2^7-8 x_1 x_2^8-250 x_1^5 x_2^3-1275 x_1^4 x_2^4 + \cdots 
+1531250 x_1^2 x_2-1071875 x_1^2. \end{small} $$
As the vector $c$ ranges over $\ZZ^6$, the degree of this curve can be arbitrarily large.
For non-rational $c$, the entropic curve is no longer algebraic. This is a general feature of
toric geometry.

Note that ${\rm pos}(A)$ is the cone over a triangular prism,
and $c$ determines a triangulation of that prism into three tetrahedra.
There are six such triangulations, one for each vertex of $P_{A,b}$.
Think of the triangulation as the union of three $\PP^3$'s in $\PP^5$.
Regularization replaces the triangulation by a nearby smooth variety.
For the entropic regularization, this is a  Segre variety
$\PP^1 \times \PP^2$. For the log-barrier regularization, 
it is the reciprocal linear space 
for~$A$.
 \hfill $\diamond$   \end{example}

The  distinction between our two regularizations mirrors
that between toric geometry and matroid theory. In statistics, this is the
distinction between {\em toric models} and {\em linear models} \cite[Section 1.2]{ASCB}.
These objects are central in the study of {\em positive geometries}
in~combinatorics and physics (cf.~\cite[Section 6]{StuTel}).
 In Section~\ref{sec2} we develop a comparative theory.
After a review of known facts in Proposition \ref{prop:volume} and \ref{prop:homeo},
we present our findings in Theorem \ref{thm:degreebounds}
and~\ref{thm:homeolambda}. They concern the algebraic curves and positive varieties
arising from linear programming.

In Section~\ref{sec3} we turn to the optimal transport problem. This is
ubiquitous in data science, where entropic regularization is a method of choice \cite{Cuturi}. Indeed, in this context the entropy function is preferred over the logarithmic barrier for efficiency reasons. We will come back to this 
preference in Remark \ref{rmk:efficiency}.
Geometrically, $P_{A,b}$ is a transportation polytope, and Segre varieties regularize
triangulations of products of simplices, as seen in Example \ref{ex:trans23}.
Our contribution is an extension  of this theory to the unbalanced regime,
which was studied in \cite{CPSV, CPSV2}.
We formulate the discrete conic coupling in eqn.~(\ref{eq:UOT}), in the spirit
 of \cite{MR3763404}.

Section~\ref{sec4} is devoted to the toric geometry and combinatorics of our new variant.
The main result is a formula for the algebraic degree of conic optimal transport
 (Theorem~\ref{thm:degreeformula}).
 In Section \ref{sec5} we discuss numerical algorithms for the
entropic regularization (\ref{eq:entropicLP}).
The task is to compute the points $x^*(\epsilon)$
along the entropic curve $\mathcal{C}^T_{A,b,c}$, and
to solve (\ref{eq:LP}) by letting $\epsilon \rightarrow 0$.

\begin{remark}
After completing this paper, we learned that the usage of the term {\em entropic barrier}
varies across the literature.  There is a general definition for arbitrary convex bodies,
due to Bubeck and Eldan. When restricted to polytopes, this leads to
the logarithmic barrier and the analytic center. This connection was developed 
from the perspective of tropical geometry by
 Allamigeon et al.~in \cite{Alla}. Their entropic path agrees with the central path, arising from
 $H(t) = - {\rm log}(t)$. What we call the entropic curve arises from $H(t) = t \cdot {\rm log}(t) - t$.
Emphasizing this distinction is important, also because we are now writing a 
``nonabelian sequel'' to the present paper, namely 
on entropic regularization of semidefinite programming.
\end{remark}

\section{Varieties and Positivity}
\label{sec2}

Let $A$ be a $d \times n$ matrix  of rank $d$ with nonnegative integer entries and no zero column.
We write $L_A$ for the row space of $A$ in $\RR^n$.
We associate two affine algebraic varieties with the matrix $A$.
Both have strong positivity properties that makes them relevant for statistics and optimization.
The {\em reciprocal linear space} $R_A$ is the Zariski closure in 
$\CC^n$ of the set of points $v^{-1} = (v_1^{-1},\ldots,v_n^{-1})$
where $v$ ranges over vectors in $L_A$ whose $n$ coordinates are nonzero.
The {\em toric variety} $T_A$ is the Zariski cosure in 
$\CC^n$ of the set of points
${\rm exp}(v) = ({\rm exp}(v_1),\ldots,{\rm exp}(v_n))$
where $v$ ranges over $L_A$. Both $R_A$ and
$T_A$ are irreducible varieties of dimension $d$, defined over
the field $\QQ$ of rational numbers.
Their prime ideals  live in the polynomial ring $\QQ[x_1,\ldots,x_n]$.

The prime ideal of $R_A$ has a distinguished universal
Gr\"obner basis. It consists of the circuit polynomials.
A circuit of $A$ is a non-zero vector $u$ of minimal support in ${\rm kernel}(A)$,
assumed to have relatively prime integer coordinates. The corresponding
circuit polynomial is the numerator of the rational function $\sum_{i=1}^n u_i/x_i$.
This is due to Proudfoot and Speyer (cf.~\cite[Proposition 12]{DSV}).
The prime ideal of $T_A$ is a toric ideal. It is generated by binomials
$$ x^{u_+} - x^{u_-}  \,\,\, = \,\,
\prod_{i: u_i > 0} \! x_i^{u_i} \,\,-\, \prod_{j: u_j < 0}  \! x_j^{-u_j}, $$
where $u$ runs over a finite set of integer vectors in ${\rm kernel}(A)$.
This set is known  in statistics as a {\em Markov basis} for the matrix $A$.
Here it usually does not  suffice to consider only circuits.

We record the well-known formulas for the degrees of our two $d$-dimensional varieties. 
In what follows we use the notation ${\rm conv}(A) \subset \mathbb{R}^d$ for the convex hull of the columns of $A$ viewed as points in $\mathbb{R}^d$, and ${\rm conv}(A \cup 0) \subset \mathbb{R}^d$ for the convex hull of ${\rm conv}(A)$ and the origin. 

\begin{proposition} \label{prop:volume}
The degree of the reciprocal linear space $R_A$ is the
M\"obius number of the rank $d$ matroid defined by the matrix $A$.
This is bounded above by $\binom{n-1}{d-1}$, with equality when
all $d \times d$ minors of $A$ are non-zero.
The degree of the toric variety $T_A$ equals the
normalized volume of the lattice polytope
${\rm conv}(A \cup 0)$. There is no upper bound in terms
of $d$ and $n$.
\end{proposition}

We refer to \cite[Section 3]{DSV} for the definition of
the M\"obius number. The fact that it gives the degree of $R_A$
follows from the result of Proudfoot and Speyer stated above.
The formula for the degree of an affine toric variety can be found in any
textbook on toric geometry.
For both varieties,
consider the semialgebraic set of points with nonnegative real coordinates:
\begin{equation}
\label{eq:semialgsets}
R^+_A \,:= \,  R_A \,\cap \, \RR^n_{\geq 0} \qquad {\rm and} \qquad
T^+_A \,:= \,  T_A \,\cap \, \RR^n_{\geq 0}.
\end{equation}
Our hypotheses on $A$ ensure that these sets are Zariski dense in $R_A$ and $T_A$
respectively, so they have dimension $d$ as well.
We now identify $A$ with the linear map $\RR^n \rightarrow \RR^d, \,v \mapsto A v$.

\begin{proposition} \label{prop:homeo}
Restricting the linear map $A$ to the two sets in (\ref{eq:semialgsets})
defines homeomorphisms
\begin{equation}
\label{eq:homeo}
R^+_A \,\simeq \,{\rm pos}(A) \qquad {\rm and} \qquad T^+_A \,\simeq \,{\rm pos}(A). 
\end{equation}
The inverse map from  the polyhedral cone on the right to
the positive variety $R^+_A$ 
resp.~$T^+_A$ on the left
takes $\,b \in {\rm pos}(A)\,$ to the analytic center resp.~Birch point of the polytope $P_{A,b}$.
\end{proposition}

\begin{proof}
For each scenario, 
consider the map that takes $b \in {\rm pos}(A)$ to the point
$x^*(\infty)$  in $P_{A,b}$. This was defined in the Introduction as
the solution to a convex optimization problem whose critical equations are polynomials.
The map is well-defined and algebraic in both cases. The image 
equals $T^+_A$ resp.~$R^+_A$. Furthermore, we have
$A \cdot x^*(\infty) = b$, so the composition with the linear map $A$ is the identity on ${\rm pos}(A)$.
This gives the desired homeomorphisms in (\ref{eq:homeo}).
\end{proof}

We now fix a sufficiently generic  vector $c \in \ZZ^n$ that serves
as  cost function in the linear program~(\ref{eq:LP}).
We augment the matrix $A$ by the row $c$ to obtain
 a $(d+1) \times n$ matrix $\binom{A}{c}$. This has rank $d+1$, since $c$ is generic.
Let $R_{\binom{A}{c}}$ be the associated reciprocal variety, and
let $T_{\binom{A}{c}}$ be the associated toric variety. Both of
these  live in $\CC^n$, and they have dimension $d+1$.
Propositions~\ref{prop:volume} and \ref{prop:homeo}
hold for these varieties, with $A$ replaced by $\binom{A}{c}$.
We note that $R_{\binom{A}{c}}$ was called the {\em central sheet}
in \cite{DSV}. Its degree  was computed 
in \cite[Theorem 11]{DSV}: it is the
M\"obius number $|\mu(A,c)|$.
By contrast, Proposition
\ref{prop:volume} refers to the M\"obius number $|\mu(A)|$.

 The toric variety $T_{\binom{A}{c}}$ is the total space
of the Gr\"obner degeneration of $T_A$  given by~$c$, as in \cite[Section 9.4]{DRS}.
The degree of $T_{\binom{A}{c}}$ is the normalized volume
of the convex hull of the $n$ columns of $\binom{A}{c}$ together with the origin in $\RR^{d+1}$.
This volume is a subtle invariant which  incorporates both geometric and arithmetic properties of 
the integer entries of $A$ and $c$.

\begin{example}[$d=2,n=4$] We consider the matrix 
$A \, = \,\begin{small} \begin{pmatrix} 3 & \! 2 & \! 1 & \! 0 \\
0 &  1 \! & \! 2 & \! 3  \end{pmatrix} \end{small} $.
In our set-up, $T_A$ is a toric surface in $\CC^4$, namely the cone over the  {\em twisted cubic curve}.
Its prime ideal is $ \langle x_1 x_3 - x_2^2, x_1 x_4 - x_2 x_3, x_2 x_4-x_3^2 \rangle$.
The reciprocal surface $R_A$ happens to be isomorphic to $T_A$.
Its prime ideal is
$
\langle 
x_1 x_2-3 x_1 x_4+2 x_2 x_4,
2 x_1 x_3-3 x_1 x_4+x_3 x_4,
x_2 x_3-2 x_2 x_4+x_3 x_4
\rangle
$.

We now augment $A$ by the cost vector $c = (c_1,c_2,c_3,c_4)$.
The resulting varieties are
 hypersurfaces in $\CC^4$.
The reciprocal variety $R_{\binom{A}{c}}$ is the affine cubic threefold defined by
\begin{equation}
\label{eq:rec3fold}
\frac{x_1x_2x_3x_4}{3} \cdot {\rm det} \begin{pmatrix} A \\ c \\ x^{-1} \end{pmatrix} \quad =  \quad
\begin{matrix}
\,(c_1-3 c_3+2 c_4) x_1 x_3 x_4 \,+\, (c_1-2 c_2+c_3)x_1x_2 x_3 \\
-(c_2-2c_3+c_4) x_2 x_3 x_4 \, \,- (2 c_1-3 c_2+c_4) x_1 x_2 x_4 
. \end{matrix}
\end{equation}
The toric variety $T_{{\binom{A}{c}}}$ is an affine threefold in $\CC^4$, defined by an
irreducible binomial such~as
\begin{equation}
\label{eq:tor3fold}
x_2^{c_1-3 c_3+2 c_4}\, x_4^{c_1-2 c_2+c_3}
 \,\,\,-\,\,\,
x_1^{c_2-2c_3+c_4} \,x_3^{2 c_1-3 c_2+c_4}.
\end{equation}
The coefficients in (\ref{eq:rec3fold}) are the
exponents in (\ref{eq:tor3fold}). The equation (\ref{eq:tor3fold}) is correct if and only if
these exponents are relatively prime and nonnegative. In that case
the degree of $T_{\binom{A}{c}}$  equals
$2(c_1 -  c_2 -  c_3 +  c_4)$. Thus the degree depends on sign conditions and divisibilities in $c$.
 \hfill $\diamond$  \end{example}

We now define the curves of interest in linear programming
by intersecting our varieties with the affine-linear spaces $\{x \in \CC^n:Ax = b\}$,
for $b \in \RR^d$. The resulting curves are denoted
\begin{equation}
\label{eq:RTcurves}
\mathcal{C}^R_{A,b,c} \,= \, R_{\binom{A}{c}} \,\cap \,\{x :Ax = b\}
\qquad {\rm and} \qquad \mathcal{C}^T_{A,b,c} \,= \, T_{\binom{A}{c}} \,\cap \,\{x :Ax = b\}. \qquad
\end{equation}

\begin{theorem}  \label{thm:degreebounds} For generic  vectors $b \in \mathbb{R}^d$ and $c \in \mathbb{Z}^n$, the intersections
in (\ref{eq:RTcurves}) are curves in~$\CC^n$,
namely the central curve and the entropic curve of the LP (\ref{eq:LP}). Their degrees satisfy
\begin{equation}
\label{eq:degreebounds}
 \begin{matrix} {\rm degree}(\mathcal{C}^R_{A,b,c}) \,=\, |\mu(A,c)| \,\leq\, \binom{n-1}{d} \quad {\rm and} \quad
 {\rm degree}(\mathcal{C}^T_{A,b,c}) \,\, \leq \,\, {\rm vol}( {\rm conv}(\binom{A}{c} \cup 0)). 
 \end{matrix} 
 \end{equation}
\end{theorem}

\begin{proof}
The formula for the degree of the central curve $\mathcal{C}^R_{A,b,c}$
appears in \cite[Theorem 13]{DSV}. The upper bound is attained when all
maximal minors of the matrix $\binom{A}{c}$ are non-zero.
The entropic curve $\mathcal{C}^T_{A,b,c}$ is the
intersection of the toric variety $T_{\binom{A}{c}}$ with  $\{x: Ax=b\}$.
The degree of $T_{\binom{A}{c}}$ equals 
${\rm vol}( {\rm conv}(\binom{A}{c} \cup 0))$. Hence the inequality follows from B\'ezout's Theorem.
This inequality can be strict, even when $b$ and $c$ are generic. 
  See Proposition \ref{prop:drop}.
\end{proof}

\begin{remark} \label{rmk:trivial}
If $n = d+1$  then Theorem \ref{thm:degreebounds} is trivial because $R_{\binom{A}{c}} = T_{\binom{A}{c}} = \CC^n$.
Note that 
$P_{A,b}$ is a line segment. The curves are straight lines, and all  numbers in 
(\ref{eq:degreebounds}) are equal to~$1$. Indeed, the normalized volume of a simplex in the lattice generated by its vertices equals 1.
\end{remark}

For applications in linear programming, we restrict our curves to the positive orthant:
\begin{equation}
\label{eq:RTcurves+}
\mathcal{C}^{R,+}_{A,b,c} \,= \, R^+_{\binom{A}{c}} \,\cap \,\{x :Ax = b\}
\qquad {\rm and} \qquad \mathcal{C}^{T,+}_{A,b,c} \,= \, T^+_{\binom{A}{c}} \,\cap \,\{x :Ax = b\}. \qquad
\end{equation}
These are real algebraic curves inside the polytope $P_{A,b}$.
Following \cite{DSV},
we call $\mathcal{C}^{R,+}_{A,b,c}$ the {\em central path} of the linear program (\ref{eq:LP}),
and we call $\mathcal{C}^{T,+}_{A,b,c}$ the {\em entropic path} of (\ref{eq:LP}).
A slight distinction to \cite{ABGJ, DSV} is that our central path
travels from the vertex of $P_{A,b}$ where $c$ is minimized to the vertex where
$c$ is maximized, passing through the analytic center of $P_{A,b}$.
For instance, Figure 1 in \cite{DSV} shows all real points on the central curve. The central
path is the piece inside the shaded hexagon $P_{A,b}$. That diagram
illustrates the transportation problem in Example~\ref{ex:trans23}.

We now come to the parametrizations of our curves. These are understood
by introducing scaled versions of the varieties $R_A$ and $T_A$.
We fix a cost vector $c \in \RR^n$ which is generic in the sense that
(\ref{eq:LP}) has a unique optimal solution for all $b \in {\rm pos}(A)$.
Let $\epsilon$ be a positive real parameter, also assumed to be fixed for now.
We consider the scaling  $\frac{1}{\epsilon} c $ of the cost vector~$c$.

Fix the affine-linear subspace $\,L_A - \frac{1}{\epsilon} c\,$ of $\RR^n$.
The {\em reciprocal affine space} $R_{A,c,\epsilon}$ 
 is the Zariski closure in 
$\CC^n$ of the set of points $v^{-1} = (v_1^{-1},\ldots,v_n^{-1})$
where $v$ ranges over vectors in $L_A-\frac{1}{\epsilon}c\,$ whose $n$ coordinates are nonzero.
The {\em scaled toric variety} $T_{A,c,\epsilon}$ is the Zariski cosure in 
$\CC^n$ of the set of points ${\rm exp}(v) = ({\rm exp}(v_1),\ldots,{\rm exp}(v_n))$
where $v$ ranges over $L_A - \frac{1}{\epsilon}c$. 

Both $R_{A,c,\epsilon}$ and
$T_{A,c,\epsilon}$ are irreducible affine varieties of dimension $d$. They are defined
over appropriate subfields of the real numbers $\RR$, namely the field
$\QQ(\epsilon)$ for $R_{A,c,\epsilon}$, and the field
$\QQ(z)$ for $T_{A,c,\epsilon}$, where  $z = {\rm exp}(-1/\epsilon)$.
If  we abbreviate
$z^c =  (z^{c_1},z^{c_2},\ldots,z^{c_n})$, then
\begin{equation}
\label{eq:zcTA}
 T_{A,c,\epsilon} \,\, = \,\, z^{c} \star T_A .  
 \end{equation}
Here $\star$ denotes the Hadamard product, so $T_{A,c,\epsilon}$ is a torus
translate of our toric variety $T_A$.
We now present a generalization of
Proposition~\ref{prop:homeo}, pertaining to
  the nonnegative varieties
  \begin{equation}
\label{eq:semialgsetslambda}
R^+_{A,c,\epsilon} \,:= \,  R_{A,c,\epsilon} \,\cap \, \RR^n_{\geq 0} \qquad {\rm and} \qquad
T^+_{A,c,\epsilon} \,:= \,  T_{A,c,\epsilon} \,\cap \, \RR^n_{\geq 0}.
\end{equation}
These sets are Zariski dense in $R_{A,c,\epsilon}$ and  $T_{A,c,\epsilon}$
respectively, so they have dimension $d$.

\begin{theorem} \label{thm:homeolambda}
Restricting the linear map $A$ to the two sets in (\ref{eq:semialgsetslambda})
defines homeomorphisms
\begin{equation}
\label{eq:homeoq}
R^+_{A,c,\epsilon} \,\simeq \,{\rm pos}(A) \qquad {\rm and} \qquad T^+_{A,c,\epsilon} \,\simeq \,{\rm pos}(A). 
\end{equation}
The inverse map from  the polyhedral cone on the right to
the positive variety on the left
takes $\,b \in {\rm pos}(A)\,$ to the optimal point $\,x^*(\epsilon)\,$ of
(\ref{eq:entropicLP}), where $\,H(t) = {\rm log}(t)$ resp.~$H(t) = t  \cdot{\rm log}(t) - t$.
For $\epsilon \rightarrow 0$, the homeomorphism approaches the
regular triangulation of ${\rm pos}(A)$ given by~$c$.
\end{theorem}

\begin{proof}
The strict convexity of the objective function in (\ref{eq:entropicLP})
ensures that the optimal solution $x^*(\epsilon)$ 
is the unique critical point of that function in $P_{A,b}$. The critical equations
are those that define our varieties, and hence
the singleton $\{x^*(\epsilon)\}$ is equal to
$ R^+_{A,c,\epsilon}  \cap  P_{A,b} \,$ resp.~$T^+_{A,c,\epsilon}  \cap  P_{A,b}$.
These two singletons are different, but they both converge to the same optimal
vertex $x^*(0)$ of (\ref{eq:LP}). The regular triangulation given by $c$ is
given combinatorially by the optimal bases as $b$ ranges over ${\rm pos}(A)$.
Each optimal basis specifies a $d$-dimensional face of the orthant $\RR^n_{\geq 0}$, and the
images of these cones  triangulate ${\rm pos}(A)$.
Both semialgebraic sets $R^{+}_{A,c,\epsilon}$ and $T^{+}_{A,c,\epsilon}$ converge,
in the Hausdorff sense, to the fan that consists of these faces of $\RR^n_{\geq 0}$.
The linear map $A$ induces a piecewise-linear isomorphism between that fan and the cone ${\rm pos}(A)$.
\end{proof}

\section{Optimal Transport}
\label{sec3}

This section features a case study that is inspired by 
applications in machine learning \cite{Cuturi, KR}.
The classical Monge optimal transportation (OT) problem 
  deals with the construction of optimal
couplings for two given probability distributions. We explain how this problem, in its simplest version, can be written as a linear program \eqref{eq:LP}. Many generalizations can be treated analogously; see e.g.~\cite{DM,GO}. In Subsection \ref{subsec32}
we carry this out for unbalanced~OT.

\subsection{The Classical Case}

Given  probability distributions  $\mu \in \RR_{\geq 0}^{d_1}$ and $\nu \in \RR_{\geq 0}^{d_2}$  on the finite sets 
$[d_1]= \{1, \ldots, d_1\}$ and $[d_2]= \{1, \ldots , d_2\}$, and a cost 
matrix $c=(c_{\kk,\ll})_{\kk \in [d_1], \ll \in [d_2]}\in \RR^{d_1\times d_2}$,  we aim to  
\begin{equation}
\label{eq:OT} 
{\rm minimize}  \sum_{(\kk,\ll) \in [d_1] \times [d_2]} \!\!\! c_{\kk,\ll} \cdot x_{\kk,\ll} 
\quad \hbox{subject to} \quad x \geq 0 \quad {\rm and}
\end{equation}
\begin{equation}\label{cond:coupconstr}
\sum_{\ll\in [d_2]} \! x_{\kk,\ll} \,=\, \mu_{\kk} \,\,\,\hbox{for all}\,\, \kk \in [d_1]
\quad \hbox{and}  \quad   \sum_{\kk\in [d_1]} \! x_{\kk,\ll} \,= \,\nu_{\ll} \,\,\,\hbox{for all} \,\, \ll \in [d_2].
\end{equation}
We interpret $\mu_\kk$ as the proportion of units of a product stored at $\kk \in [d_1]$ and $\nu_\ll$ as the proportion of units desired at $\ll \in [d_2]$. Our goal is to transport all units from $[d_1]$ to $[d_2]$ with minimal transportation cost. 
The entry $c_{\kk,\ll}$ is the cost of transporting one unit from $\kk$ to $\ll$.
The feasible solutions $x = (x_{\kk,\ll})$ are known as {\em transportation plans},
 or as {\em couplings} of  $\mu$ and $\nu$.
Since  $\|\mu\|_1= \|\nu\|_1=1$,  any solution $x$ is a probability distribution on $[d_1] \times [d_2]$.

 The matrix $A$ for the linear program above has $d=d_1+d_2-1$ rows and $n = d_1 d_2$  columns,
 and its entries are in $\{0,1\}$. It represents the linear map
 that takes a $d_1 \times d_2$ matrix $x$ to its vector 
 $b = (\mu,\nu)$ of row sums and column sums. 
 Here $\nu_{d_2}$ is deleted, so   the rows of $A$ are linearly independent.
 In OT theory it is customary to keep this redundancy.
 We saw the matrix $A$   for  $d_1=2, d_2 = 3$ in (\ref{eq:A23}).
The feasible region $P_{A,b}$ is a 
 {\em transportation polytope}, consisting of
all nonnegative $d_1 \times d_2$ matrices with fixed row and column sums.
Every transportation polytope contains a unique rank one matrix $x$, namely
 the Birch point $x= (\mu_\kk \cdot \nu_\ll)$ of $P_{A,b}$. This corresponds to an independent
joint distribution.

The polytope underlying the cone ${\rm pos}(A)$ is 
 the product $\Delta_{d_1-1} \times \Delta_{d_2-1}$ of two simplices.
 The triangulations of {\rm pos}(A) are studied in
 \cite[Section 6.2]{DRS}. The  toric variety $T_A$  is the
 cone over the Segre variety $\PP^{d_1-1} \times \PP^{d_2-1}$.
 Its points are the
   $d_1 \times d_2$ matrices of rank at most $1$. 
    The prime ideal of $T_A$ is generated by the
 $2 \times 2$ minors of
a $d_1 \times d_2$ matrix; see
 \cite[Example~5.1]{GBCP}.
The positive variety $T_A^+$ represents the independence model
for distributions~on~$[d_1]$ and $[d_2]$.
We know from Proposition \ref{prop:homeo} that the linear map $A$
identifies $T_A^+$ with the cone ${\rm pos}(A)$.

The same holds for the positive part $R_A^+$ of the reciprocal
variety $R_A$. From a combinatorial perspective, it would be interesting to study 
this variety for OT in more detail.
However, in the remainder of this paper we focus on the
toric variety $T_A$ instead. Here is the reason:

\begin{remark} \label{rmk:efficiency}
In machine learning one uses \emph{entropic} regularization rather than \emph{logarithmic barrier}  regularization in~(\ref{eq:entropicLP}). The former is more efficient than the latter.
 Thus, when $d_1$ and $d_2$ are large,
the entropic path $\mathcal{C}^{T,+}_{A,b,c}$ is preferred to the
central path $\mathcal{C}^{R,+}_{A,b,c}$. 
We refer to \cite{Cuturi} for an explanation. Example \ref{ex:sinkhorn}
and the introduction of \cite{Weed} offer details and references.
\end{remark}

We next explain the degree drop which was observed for the entropic curve in Example~\ref{ex:trans23}.

\begin{proposition}  \label{prop:drop}
Let $b \in {\rm pos}(A)$ and $c \in \ZZ^{d_1 \times d_2} $
where $A$ is the matrix for OT. If $d_2 \geq 3$ then the upper bound 
in (\ref{eq:degreebounds}) for the degree  of 
the entropic curve $\mathcal{C}^T_{A,b,c}$ is always strict.
\end{proposition}

\begin{proof}
The trivial case $d_1=d_2=2$ is covered by Remark \ref{rmk:trivial}.
We have $d_2 \geq 3$, so $n = d_1 d_2 $ is larger than $ d+1 = d_1 + d_2$.
Since $T_A$ and $T_{\binom{A}{c}}$ are affine toric varieties in $\CC^n$, we consider
their closures  $\overline{T}_A$ and $\overline{T}_{\binom{A}{c}}$ in $\PP^n$.
We write $\{x_0 = 0\} $ for the hyperplane at infinity $ \PP^n \backslash \CC^n$.
We are interested in the closure in $\PP^n$
of the entropic curve. This projective curve is denoted~$\overline{ \mathcal{C}}^T_{A,b,c}$.

 The upper bound on the right in (\ref{eq:degreebounds}) is the degree of the
$(d+1)$-dimensional toric variety $\overline{T}_{\binom{A}{c}}$ in $\PP^n$.
We intersect $\overline{T}_{\binom{A}{c}}$ with the codimension $d$ linear space
$\{x \in \PP^n : A x = x_0 b \}$. One of the irreducible components of this intersection 
is the curve  $\overline{ \mathcal{C}}^T_{A,b,c}$.
By the general B\'ezout Theorem, the equation 
$\,{\rm degree}\bigl(\overline{ \,\mathcal{C}}^T_{A,b,c}\bigr) = 
{\rm degree}\bigl(\,\overline{T}_{\binom{A}{c}}\bigr)\,$ means that 
there is no component other than the entropic curve.
Our goal is therefore to identify an extraneous component~in 
\begin{equation}
\label{eq:specialintersection1}
\overline{T}_{\binom{A}{c}} \,\, \cap \,\, \bigl\{ x \in \PP^n : A x = x_0 b \bigr\} . 
\end{equation}
Restricting to the hyperplane at infinity, we see that (\ref{eq:specialintersection1}) contains
\begin{equation}
\label{eq:specialintersection2}
\overline{T}_A \,\cap\, \bigl\{ x \in \PP^n : A x = 0 \bigr\}  \quad \supseteq \quad
  T_A \,\cap \, \bigl\{x \in \CC^n : Ax = 0 \bigr\}.
  \end{equation}
The affine variety on the right consists of all
$d_1 \times d_2$ matrices of rank $\leq 1$ whose
rows and columns sum to zero. Such matrices have the form
$x = (\mu_\kappa \cdot \nu_\lambda)$ where $\mu \in \CC^{d_1}$ and $\nu \in \CC^{d_2}$
satisfy $\sum_{\kappa=1}^{d_1} \mu_\kappa    = \sum_{\lambda=1}^{d_2} \nu_\lambda  = 0$.
This variety has dimension $d_1 + d_2- 3 \geq 2$, so the
intersection (\ref{eq:specialintersection1}) has an extraneous component whose
dimension exceeds that of $\,\mathcal{C}^T_{A,b,c}$.
\end{proof}

\begin{remark}
Our proof reflects the special behavior we already know from the intersection
\begin{equation}
\label{eq:specialintersection3}
\overline{T}_A \,\cap\, \bigl\{ x \in \PP^n : A x = x_0 b \bigr\}  \quad \supseteq \quad
  T_A \,\cap \, \bigl\{x \in \CC^n : Ax = b \bigr\}.
  \end{equation}
  The toric variety $T_A$ has degree $\binom{d_1+d_2-2}{d_1-1}$,
  but the intersection on the right has degree one.
  It is a single point, which is
    rational over $b = (\mu,\nu)$, namely the Birch point
  $x = (\mu_\kappa \cdot \nu_\lambda)$.
  \end{remark}

\subsection{Unbalanced Case: Conic Coupling} 
\label{subsec32}

Problem (\ref{eq:OT}) is infeasible for optimal transport between measures $\mu$ and $\nu$ with $\|\mu\|_1 \ne \|\nu\|_1 $. This unbalanced case is relevant in  the statistical analysis of partial or incomplete data sets.  
One remedy is to replace the 
hard constraint \eqref{cond:coupconstr} by
a penalty function, e.g.\ Kullback-Leibler \cite{CPSV}.
We here follow \cite{CPSV2, MR3763404} and present a
  linear programming formulation (\ref{eq:LP}).
   In particular, this formulation can be understood as a moment constrained optimal transport problem.

  Let us assume that, after discretization and scaling, the entries of
  the margins $\mu$ and $\nu$ are integers. This can be achieved up to arbitrary numerical precision. More precisely, we fix positive integers
  $e_1$ and $e_2$ such that 
  $\mu_\kappa \in  [e_1]$ for all  $\kappa \in [d_1]$
  and $\nu_\lambda \in [e_2]$ for all $\lambda \in [d_2]$.
  
  We fix the state spaces $[d_1] {\times} [e_1]$ and  $[d_2] {\times} [e_2]$.
A joint probability distribution $x = (x_{\kk, i, \ll ,j})$ on their product 
$ ([d_1] {\times} [e_1]) \times  ([d_2] {\times} [e_2])$ is called
    a \textit{conic coupling} for $\mu$ and $\nu$ if 
\begin{equation}\label{cond:gencoup}
 \sum_{\ll =1}^{d_2} \sum_{i=1}^{e_1} \sum_{j =1}^{e_2} i \,x_{\kk,i,\ll,j} = \mu_{\kk}\,\,\, {\rm for}\,\, \kk \in [d_1]
 \quad {\rm and}  \quad  
 \sum_{\kk =1}^{d_1} \sum_{i=1}^{e_1} \sum_{j=1}^{e_2} j \,x_{\kk,i,\ll,j} = \nu_{\ll}\,\,\, {\rm for}\,\, \ll  \in [d_2].
\end{equation}
We also assume that the cost function is extended to $c:
 ([d_1] {\times} [e_1]) \times  ([d_2] {\times} [e_2]) \rightarrow \RR$. 
  The value $c_{\kk, i, \ll, j}$ is interpreted as the cost of generating
   $j$ units of mass at $\ll \in [d_2]$ from $i$ units of mass at $\kk\in [d_1]$.
       We propose the following relaxation of  OT in the unbalanced case:
\begin{equation}
\label{eq:UOT}
{\rm Minimize}  \sum_{(\kk,i,\ll,j) \,\,\in \atop 
[d_1] {\times} [e_1] \times [d_2] {\times} [e_2]} \!\!\!\!\!\!\!
c_{\kk,i,\ll,j} \cdot x_{\kk,i,\ll,j} \,\,\hbox{ subject to }\,\, x \geq 0 \,\,\,{\rm and}\, \,\, (\ref{cond:gencoup}).
  \end{equation}
In the context of statistics, one can (but need not) impose the  normalization constraint
\begin{equation}
\label{eq:sumtoone}
\sum_{(\kk,i,\ll,j) \,\,\in \atop 
[d_1] {\times} [e_1] \times [d_2] {\times} [e_2]} 
 \!\!\!\!\!\!\! x_{\kk,i,\ll,j} \quad = \quad 1 .
 \end{equation}
The minimizers $x$ for the problem  (\ref{eq:UOT})-(\ref{eq:sumtoone}) are
 called \textit{optimal conic couplings} of $\mu$ and~$\nu$. They define a  cost-optimal random sampling mechanism of particle cluster pairs in $[d_1]$ and $[d_2]$ whose  mean marginal empirical distributions
  are  $\mu$ and $\nu$, respectively.  
We next show that our formulation makes sense, meaning that conic couplings always exist.

\begin{lemma}
The linear program  (\ref{eq:UOT})-(\ref{eq:sumtoone}) is feasible
for all $\,\mu \in [e_1]^{d_1}$ and all $\,\nu \in [e_2]^{d_2}$.
\end{lemma}

\begin{proof}
Let $\overline{\mu} = \frac{1}{ \|\mu\|_1} \mu $ and
$\overline{\nu} = \frac{1}{ \|\nu\|_1} \nu $ be the induced probability distributions 
on $[d_1]$ and~$[d_2]$.
We define a probability distribution $x =(x_{\kk,i,\ll,j})$ on the space
$\, [d_1] {\times} [e_1] \times [d_2] {\times} [e_2]$ by setting
\begin{equation}
\label{eq:bysetting}
 x_{\kk,i,\ll,j} \,\,\,=\,\,\,\,  \overline  \mu_\kk \cdot
  \delta_{\|\mu\|_1,i} \,\cdot \, \overline \nu_\ll  \cdot\delta_{\|\nu\|_1,j}.
\end{equation}
Here we use Kronecker delta notation, i.e.~$\delta_{a,b} = 1$ if $a=b\,$ and
$\,\delta_{a,b} = 0$ if $a \not= b$.
The numbers in (\ref{eq:bysetting}) are nonnegative. One checks that they satisfy both
(\ref{cond:gencoup}) and (\ref{eq:sumtoone}).
\end{proof}




To connect to our general set up we write the linear program (\ref{eq:UOT})  in the standard form~(\ref{eq:LP}).
In what follows we assume that  $d_1,d_2,e_1,e_2 \geq 2$. 
The matrix $A$ has $n=d_1e_1 d_2 e_2$ columns and
 $d = d_1 + d_2$ linearly independent rows.
 We identify $\CC^n$ with the space
of tensors $x = (x_{\kappa,i,\lambda,j})$ of
  format $d_1 {\times} e_1 \times d_2  {\times} e_2$.
  The column of $A$ indexed by $(\kappa,i,\lambda,j)$ is the
  vector  $i {\bf e}_\kappa \oplus j {\bf e}_\lambda $
in $\NN^{d} = \NN^{d_1} \oplus \NN^{d_2}$, where ${\bf e}_\kappa$ and ${\bf e}_\lambda$ denote unit vectors.
If we set $b = (\mu,\nu)^T \in \RR^d$ then the
polytope $P_{A,b}$ consists of all nonnegative tensors $x$ that satisfy
the linear constraints (\ref{cond:gencoup}).

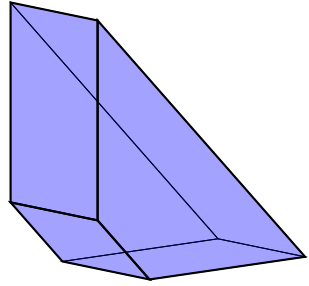
\begin{figure}[h]
\centering
\tdplotsetmaincoords{80}{50}
\begin{tikzpicture}[baseline=(A.base),scale=1.8,tdplot_main_coords]
\coordinate (A) at (0,2,0);
\coordinate (B) at (0,0,2);
\coordinate (C) at (0,1/2,0);
\coordinate (D) at (0,0,1/2);
\coordinate (E) at (1,2,0);
\coordinate (F) at (1,0,2);
\coordinate (G) at (1,1/2,0);
\coordinate (H) at (1,0,1/2);
	
\draw[fill opacity=0.2,fill = blue,] (A)--(B)--(F)--(E)--cycle;		
\draw[fill opacity=0.2,fill = blue,] (A)--(B)--(D)--(C)--cycle;
\draw[fill opacity=0.2,fill = blue] (E)--(G)--(C)--(A)--cycle;
\draw[fill opacity=0.2,fill = blue,thick] (C)--(D)--(H)--(G)--cycle;
\draw[fill opacity=0.2,fill = blue,thick] (B)--(D)--(H)--(F)--cycle;
\draw[fill opacity=0.2,fill = blue,thick] (E)--(F)--(H)--(G)--cycle;
\end{tikzpicture}
\caption{The $4$-dimensional cone ${\rm pos}(A)$ in Example \ref{ex:2222} has a slanted cube for its base.}
\label{fig:P2222}
\end{figure}

\begin{example}[$d_1 = e_1 = d_2 = e_2 = 2$] \label{ex:2222} \setcounter{MaxMatrixCols}{20}
Our matrix has $d=4$ rows and $n=16$ columns:
\begin{equation}
\label{eq:4by16}
 A \,\, = \,\,
 \begin{small}
\begin{pmatrix}
1 & 1 & 1 & 1 & 2 & 2 & 2 & 2 & 0 & 0 & 0 & 0 & 0 & 0 & 0 & 0 \\
0 & 0 & 0 & 0 & 0 & 0 & 0 & 0 & 1 & 1 & 1 & 1 & 2 & 2 & 2 & 2 \\
1 & 2 & 0 & 0 & 1 & 2 & 0 & 0 & 1 & 2 & 0 & 0 & 1 & 2 & 0 & 0 \\
0 & 0 & 1 & 2 & 0 & 0 & 1 & 2 & 0 & 0 & 1 & 2 & 0 & 0 & 1 & 2 \\
\end{pmatrix}. 
\end{small}
\end{equation}
We identify $\CC^{16}$ with the space of
$2 \times 2 \times 2 \times 2$-tensors $x = (x_{\kappa,i,\lambda,j})$.
The coordinates  
$x_{1111}, x_{1112}, \ldots,x_{2222}$
are ordered lexicographically, which matches the column ordering of~$A$.
The  toric variety $X_A$ has dimension $4$ and degree $72$
  in $\CC^{16}$.
The prime ideal of $X_A$ is homogeneous
with respect to the column sum grading  
$(2,3,2,3,3,4,3,4,2,3,2,3,3,4,3,4)$.
It is minimally generated by $39$ binomials: $5$ of degree $4$,
$8$ of degree $5$, $18$ of degree~$6$,
and $8$ of degree~$7$.
The polyhedral cone $ {\rm pos}(A)$ is spanned by $8$ rays, and it has $6$ facets. Explicitly, 
\begin{equation}
\label{eq:coneover3cube}
 {\rm pos}(A) \,\,= \,\, \bigl\{\, b \in \RR^4_{\geq 0} \,:\,
b_1 + b_2 \,\leq\, 2b_3 + 2b_4 \,\,\, {\rm and} \,\,\, b_3 + b_4 \,\leq\,  2b_1 + 2b_2 \bigr\}.
\end{equation}
This is the cone over a polytope combinatorially isomorphic to a $3$-cube,
 shown in Figure~\ref{fig:P2222}.  The vertices of that cube correspond to 
 the eight columns of $A$ with entries $0,0,1,2$.
 \hfill $\diamond$   \end{example}

\section{Polytopes and their Volumes}
\label{sec4}

The entropic method for solving the linear program \eqref{eq:LP} is a two-step process.
First, the solution $x^*(\epsilon)$ to the regularized problem \eqref{eq:entropicLP}  is computed.  Here,
 $\epsilon > 0$ and
$H(t) = t \, \log(t) - t$.
Second, one lets $\epsilon \rightarrow 0$ and tracks the minimizer $x^*(\epsilon)$ to the optimal vertex $x^*(0)$ of~$P_{A,b}$. 

Step 1 amounts to solving the polynomial system given by $A \, x = b$ and $x \in T_{A,c,\epsilon}$.
For linear programming, one wants the unique positive solution $x^*(\epsilon)$.
But, for other applications, e.g.~scattering amplitudes in particle physics \cite{StuTel},
all complex solutions are needed.  A standard method for finding them all
 is \emph{homotopy continuation} \cite{sommese2005numerical}. 
 We  expect the number of solutions  to be
$\deg T_{A,c,\epsilon} = {\rm vol}({\rm conv}(A \cup 0))$, and this
is the number of paths to be tracked. This number is also the algebraic degree
of $x^*(\epsilon)$, over the ground field $\QQ(z)$ in~(\ref{eq:zcTA}).
   
Numerical algebraic geometry interfaces
gracefully with interior point methods in optimization.
    In a scenario where the matrix $A$ is fixed and
    \eqref{eq:entropicLP} must be solved for many different vectors $b$ and $c$,
     it makes sense to initialize by computing all complex solutions.
This needs to be done     \emph{only once}. Indeed, for new parameters $b', c'$, one can use $x^*(\epsilon) \in T_{A,c,\epsilon} \cap \{A \, x = b \}$ as a start solution to find the positive point in $T_{A,c',\epsilon} \cap \{ A \, x = b' \}$.  We will come back to continuation methods at the
end of Section \ref{sec5},
 in our discussion of step 2, in which $\epsilon \rightarrow 0$. 

Given an interesting matrix $A$, the reasons above motivate the
combinatorial problem~of finding the degree of $T_{A,c,\epsilon}$.
This means finding the volume of the polytope ${\rm conv}(A \cup 0)$.
We here solve this problem for unbalanced optimal transport,
as formulated in Subsection~\ref{subsec32}.

Let $A$ be the $d \times n$ matrix for conic coupling (\ref{eq:UOT}),
where $d = d_1 + d_2$ and $n=d_1e_1 d_2 e_2$.
For any right hand side 
$b = (b_1,\ldots,b_{d_1},b_{d_1+1},\ldots,b_{d_2})^T$,
the set of feasible solutions is the polytope $P_{A,b}$. We know that
$P_{A,b} \not= \emptyset $ if and only if $b \in {\rm pos}(A)$,
and ${\rm dim}(P_{A,b}) = n-d$ if and only if $b$ is in the
interior of ${\rm pos}(A)$. Our next result characterizes that cone, as in~(\ref{eq:coneover3cube}).

\begin{proposition} \label{prop:feasibility}
The feasibility cone ${\rm pos}(A)$ for the conic coupling problem (\ref{eq:UOT}) equals
 $$
 \bigl\{ \,y \in \RR^d_{\geq 0} \,:\,
  y_1 + \cdots + y_{d_1} \leq e_1(y_{d_1+1} + \cdots + y_{d_1+d_2}) ,\,
   y_{d_1+1} + \cdots + y_{d_1+d_2} \leq e_2(y_1 + \cdots + y_{d_1}) \bigr\}. 
$$
This $d$-dimensional cone has $2 d_1 d_2$ rays and $d_1+d_2 + 2$ facets. It is the cone over
 a simple $(d-1)$-dimensional polytope which is combinatorially isomorphic to
the product of simplices
$$ \Delta_1 \times \Delta_{d_1-1} \times \Delta_{d_2-1}.$$
\end{proposition}

\begin{proof}
Let $K$ be the polyhedral cone given in the assertion.
Every column vector $i {\bf e}_\kappa \oplus j {\bf e}_\lambda$ of the matrix $A$
lies in $K$ because $0 \leq i \leq j e_1$ and $0 \leq j \leq i e_2$. Hence ${\rm pos}(A) 
\subseteq K$. For the reverse inclusion, we identify the extreme rays of $K$.
Every vector in $K$ must have at least one positive coordinate among the
first $d_1$ coordinates and ditto for the last $d_2$ coordinates. We see that at most
$d-2$ of the nonnegativity constraints can be attained. Thus every extreme ray
must attain equality in at least one of the other inequalities. This implies that the extreme rays
are $ {\bf e}_\kappa \oplus e_2 {\bf e}_\lambda $ for some $\kappa \in [d_1]$ and
$e_1 {\bf e}_\kappa \oplus  {\bf e}_\lambda $ for some $\lambda \in [d_2]$.
\end{proof}

The following  result pertains to the affine variety  $T_A$.
 Its proof is analogous to that above.
  
 \begin{proposition} \label{prop:facets}
 The $d$-dimensional polytope ${\rm conv}( A \cup 0)$  has $d+4$ facets, given by the $d + 2$ inequalities defining 
 ${\rm pos}(A)$, together with $y_1 + \cdots + y_{d_1} \leq e_1$ and $y_{d_1+1} + \cdots + y_{d_1+d_2} \leq e_2$. 
 \end{proposition}

Solving the entropic regularization (\ref{eq:entropicLP}) for (\ref{eq:UOT}) means intersecting the
polytope $P_{A,b}$ with the
scaled toric variety $T_{A,c,\epsilon} = z^{c} \star T_A$, where $z = {\rm exp}(-1/\epsilon)$.
Algebraically, we compute the unique positive solution $x = x^*(\epsilon)$
to the following equations, with $H(t) = t \cdot {\rm log}(t) - t$:
\begin{equation}
\label{eq:birchconstraints}
 A x = b \qquad {\rm and} \qquad {\rm rank} \begin{pmatrix} A \\ c + \epsilon H'(x) \end{pmatrix} \,\leq \, d+1 .
 \end{equation}
 The algebraic degree of (\ref{eq:birchconstraints}) is the number of solutions in $\CC^n$.
This is the  degree over $\mathbb{Q}$~of the floating point numbers that are output by any numerical algorithm.
 It is bounded above~by
\begin{equation} \label{eq:degzT}
{\rm degree}(T_{A,c,\epsilon}) \,\,= \,\,{\rm degree}(T_A) 
\,\,\,=\,\,\, {\rm vol}\bigl( {\rm conv}(A \cup 0) \bigr)\,.
\end{equation}

Our main result is a formula in terms of  $d_1,e_1,d_2,e_2$ for this algebraic complexity measure.
In other words, we generalize the number $72$,
which is the degree of $T_A \subset \CC^{16}$ in Example~\ref{ex:2222}.

\begin{theorem} \label{thm:degreeformula}
The algebraic degree of the constraints (\ref{eq:birchconstraints}) for optimal conic coupling~is
\begin{equation}
\label{eq:degreeformula} 
{\rm degree}(T_A) 
\,\, = \,\,\,
\binom{d_1{+}d_2}{d_1} \biggl((e_1^{d_1}-1)(e_2^{d_2}-1) \,+\, \frac{d_1}{d_1{+}d_2}(e_2^{d_2}-1)\,+\, \frac{d_2}{d_1{+}d_2}(e_1^{d_1}-1) \biggr).
\end{equation}
\end{theorem}

To illustrate our formula, consider the binary case $(d_1=d_2=2)$, where it gives
$\,\binom{4}{2} (9 + \frac{2}{4} 3 + \frac{2}{4} 3) = 72$, and the ternary case $(d_1=d_2=3)$, where
$\, \binom{6}{3} (26^2 + \frac{3}{6} 26 + \frac{3}{6} 26) = 14040$.

\begin{proof} We compute the volume in (\ref{eq:degzT}).
Fix integers $d,e \geq 2$ and consider the $d$-polytope
$$ P_{d,e} \,\,=\,\, {\rm conv} \bigl\{ \,k {\bf e}_i \,: \, i=1,\ldots,d \,\,{\rm and}\, \, k = 1,\ldots,e \bigr\}. $$
The normalized volume of this polytope is  $ e^d-1$. 
The convex hull of the columns of $A$ equals
$$ {\rm conv}(A) \,\, = \,\,  P_{d_1,e_1} \,\times \, P_{d_2,e_2}. $$
The normalized volume of a direct product is multiplicative up to a binomial coefficient, so
\begin{equation} \begin{matrix}
\label{eq:volconA} {\rm vol}({\rm conv}(A)) \,\, = \,\, \binom{d_1+d_2}{d_1}\, {\rm vol}(P_{d_1,e_1})\, {\rm vol}(P_{d_2,e_2})\,\,=\,\,
\,\binom{d_1+d_2}{d_1} (e_1^{d_1}-1)(e_2^{d_2}-1) .
\end{matrix} \end{equation}

This explains the first summand in (\ref{eq:degreeformula}).
It remains to determine the volume of the region ${\rm conv}(A \cup 0) \backslash {\rm conv}(A)$.
To this end, we consider the facets of ${\rm conv}(A)$ that are visible from the origin $0$.
There are precisely two such facets, and they are defined respectively by
\begin{equation}
\label{eq:2hyperplanes} y_1 + \cdots + y_{d_1} = 1
\quad {\rm and} \quad y_{d_1+1} + \cdots + y_{d_2} = 1. 
\end{equation}
These two facets are the $(d_1+d_2-1)$-dimensional polytopes
$\, \Delta_{d_1-1} \times P_{d_2,e_2} \,$ and $\,
P_{d_1,e_1} \times \Delta_{d_2-1} $.
 Since the origin has lattice distance one from the hyperplanes (\ref{eq:2hyperplanes}),
the volume of the region
 $\,{\rm conv}(A \cup 0) \backslash {\rm conv}(A)\,$ coincides with the sum of the volumes 
 of the two polytopes:
 $$  \begin{matrix} {\rm vol} \bigl( \Delta_{d_1-1} \times P_{d_2,e_2}\bigr) \,+\,
{\rm vol}\bigl(P_{d_1,e_1} \times \Delta_{d_2-1}\bigr)\,= \, 
\binom{d_1{+}d_2{-}1}{d_2} (d_2^{e_2}-1) \, + \,
\binom{d_1{+}d_2{-}1}{d_1} (d_1^{e_1}-1).  
\end{matrix} $$
This gives the last two summands in (\ref{eq:degreeformula}), and the proof is complete.
\end{proof}

\section{Computational Schemes}
\label{sec5}

We now turn to convex optimization methods for solving (\ref{eq:entropicLP}). 
Recall that $H(t) = t \cdot {\rm log}(t) - t$ and hence $H^*(s) = {\rm exp}(s)$ in 
 the dual formulation. We can solve \eqref{EqDualProblem} using
coordinate ascent, i.e.~by iteratively optimizing each variable $p_i$ in \eqref{EqDualProblem} in a cyclic order. 
In statistics, this is known as \emph{iterative proportional scaling} (IPS,  see \cite{DR, she2018iterative}).
This method converges linearly \cite{LuoTseng}. 
Randomized iterations over the $p_i$ can further improve the performance. When each one-dimensional 
 optimization is computationally cheap, this method is particularly interesting. 
 
 \begin{example}[Sinkhorn iterations] \label{ex:sinkhorn}
For classical optimal transport \eqref{eq:OT}, coordinate ascent 
is the well-known \emph{Sinkhorn algorithm} \cite{AKRS, Cuturi, KR}. It uses
  highly efficient matrix-vector products.

Writing $(f_\kappa)_{\kappa \in [d_1]}$ and $(g_\lambda)_{\lambda \in [d_2]}$ for the dual variables,
the dual OT problem \eqref{EqDualProblem} reads:
\begin{equation} {\rm Maximize} \quad
    \sum_{\kappa= 1}^{d_1} \mu_\kappa f_\kappa \,+\, \sum_{\lambda= 1}^{d_2} \nu_\lambda g_\lambda 
    \,-\, \epsilon \cdot \sum_{\kappa= 1}^{d_1}\sum_{\lambda= 1}^{d_2} 
    {\rm exp}\bigl(\, (f_\kappa + g_\lambda - c_{\kappa,\lambda})/\epsilon\, \bigr).
\end{equation}
It is easy to solve this for each variable separately. 
Equating derivatives to zero, we find
\begin{equation}\label{EqExplicitSinkhornIterates}
    f_\kappa \,\,=\,\, - \,\epsilon \cdot \log\left(\sum_{\lambda = 1}^{d_2} 
    {\rm exp}\bigl( (g_\lambda - c_{\kappa,\lambda})/\epsilon \bigr) \right) \,\,+\,\, \epsilon\cdot \log(\mu_\kappa)
   \quad \hbox{and similarly for} \,\,\,g_\lambda.
\end{equation}
Sinkhorn iteration means executing these assignments.
A useful reformulation 
 is obtained by setting 
 $F_\kappa = {\rm exp}({f_\kappa/\epsilon})$, $G_\lambda = {\rm exp}({f_\lambda/\epsilon})$, and 
$K_{\kappa,\lambda} = {\rm exp}({-c_{\kappa,\lambda}/\epsilon})$. Here $F$ is a row vector, and $G$ is a column vector. With this, the rules for updating $F$ and $G$ are
$F_\kappa = \mu_\kappa / [K \cdot G]_\kappa$ 
and 
$G_\lambda = \nu_\lambda/[ F \cdot K]_\lambda$.
The primal solution is the matrix 
 $x= {\rm diag}(F)\cdot K\cdot {\rm diag}(G)$.
These steps are highly parallelizable, so
 large-scale problems  can be solved effectively. This  explains the preference for entropic regularization 
in Remark \ref{rmk:efficiency}.
 \hfill $\diamond$   \end{example}

Coordinate ascent can be applied for any matrix $A$, but in general
there is no simple formula for the one-variable updates.
But, we can resort to non-linear optimization for this.

\begin{example}[Coordinate ascent for entropic conic transport]
The dual problem for (\ref{eq:UOT})~is
\begin{equation}
{\rm Maximize} \,\,\,
    h \,\,+\, \sum_{\kappa= 1}^{d_1} \mu_\kappa f_\kappa \,+\,
     \sum_{\lambda= 1}^{d_2} \nu_\lambda g_\lambda - \epsilon 
    \cdot \sum_{\kappa,\lambda,i,j} {\rm exp}\bigl( \,(h + if_\kappa \,
    +\, jg_\lambda - c_{\kappa,i,\lambda,j})/\epsilon\, \bigr).
\end{equation}
Here we also assumed \eqref{eq:sumtoone},
and $h$ is the dual variable for that normalization constraint.
Coordinate ascent means that we compute,
for each $\kappa$, the unique positive solution $F_\kappa$ to
\begin{equation} \label{eq:univpol} 
\sum_{i = 1}^{d_1} i \cdot \gamma_{\kappa,i} \cdot (F_\kappa)^i \,=\, \mu_\kappa \,,
\end{equation}
where $\gamma_{\kappa,i} = \sum_{\lambda,j} 
{\rm exp} \bigl( (h + jg_\lambda - c_{\kappa,i,\lambda,j})/\epsilon\bigr)$.
 This step is more costly than applying \eqref{EqExplicitSinkhornIterates}.
 \hfill $\diamond$   \end{example}

Solving \eqref{eq:univpol} is costly. One prefers cheap iterations, inspired by first-order methods.
 Of special interest is the Darroch-Ratcliff algorithm \cite{DR}, 
  which is also known as \emph{generalized iterative scaling} (GIS).
  This was   recognized 
   in \cite{she2018iterative} as an instance of majorization-minimization on the dual formulation \eqref{EqDualProblem}. 
GIS is a remarkably simple iterative process. As with Sinkhorn, each step involves $d$ matrix-vector products.
See  \cite[Figure 4]{AKRS} for the connection.
Theorem~\ref{thm:DR} below shows that GIS can be used\footnote{An illustration of entropic conic unbalanced OT, for numerical comparison between GIS, IPS and general purpose convex optimization, is implemented
at \url{https://github.com/fxv27/EntropicConicUOT}} effectively
for conic coupling~\eqref{eq:UOT}-\eqref{eq:sumtoone}.

Before starting the iteration, we modify $A, b$ and $c$  slightly. To match
   \cite{DR}, we formulate an equivalent linear program where all columns of $A$ have the same sum.
For this conversion, we require that the all-ones vector $(1, \ldots, 1)$ is in the row space $L_A$. In geometric terms, this means that $T_A$ is the affine cone over a projective toric variety. The matrix $A$ for classical OT satisfies this assumption. In the unbalanced case, it holds after  we add the constraint~\eqref{eq:sumtoone}. 

We  now assume $(1, \ldots, 1) \in L_A$. Fix $b \in {\rm pos}(A)$.
Then  $s = \sum_{i=1}^n x_i$ is fixed for $x \in P_{A,b}$.
Let $a$ be the maximum among the column sums of $A$. To each column $a_j$, we append the entry $a_{d+1,j} = a - |a_j|$, where $|a_j| = \sum_{i = 1}^d a_{ij}$.
Prepending the column $(0, \ldots, 0, a)$, we obtain
\[ {\cal A}\, \,=\, \begin{bmatrix}
\,\,0 & & A & & \\
\,\,a & a_{d+1,1} & \cdots & a_{d+1,n}
\end{bmatrix} \quad \in\,\, \NN^{(d+1) \times (n+1)} .\]
Note that the entries in each column of ${\cal A}$
sum to $a$. Let $s_c = 1 + \sum_{i=1}^n \exp(-c_i/\epsilon) $, and
\[ \beta \,=\, \left( \,\frac{b}{ s + 1 } \,,\,\, a -  \frac{|b|}{ s + 1 } \, \right)^{\! \top} \quad \text{and} \quad \gamma 
\,=\, \bigl( \epsilon  \log(s_c), \, c_1 + \epsilon \log(s_c), \, \ldots, \, c_n + \epsilon \log(s_c) \bigr). \]
These data define the following variant of the regularized linear program (\ref{eq:entropicLP}):
\begin{equation} \label{eq:modifiedLP}
{\rm Minimize}\, \,\,\gamma \cdot y \,+\, \epsilon \sum_{i=0}^n H(y_i) \,\,\,\,\hbox{subject to} \,\,\,{\cal A} \, y = \beta
\,\,{\rm and} \,\, y \geq 0 . 
\end{equation}
We now rephrase the
result of Darroch and Ratcliff \cite{DR}
in the geometric setting of
 Section \ref{sec2}.
An essentially equivalent formulation was presented
recently in \cite[Proposition 5.1]{AKRS}.

\begin{theorem} \label{thm:DR}
If  \eqref{eq:entropicLP} is feasible, then the solution $x^*(\epsilon)$ is given by $(y_1/y_0, \ldots, y_n/y_0)$, where $y = y^*(\epsilon) \in \RR^{n+1}_{\geq 0}$ is the unique solution to \eqref{eq:modifiedLP}. That is, $y$ is the unique 
point in $T_{\cal A, \gamma,\epsilon}^+ \cap \{ {\cal A} \, y = \beta \}$. It satisfies $\sum_{i=0}^n y_i = 1$ and is obtained as the unique limit point of  the iteration
\begin{equation}
\label{eq:IPF} y^{(0)} \,=\, z^{\gamma} := \exp(-\gamma/\epsilon), \qquad y_i^{(k+1)} 
\,=\, y_i^{(k)} \left( \frac{\beta^{a_i}}{(\mathcal{A} \, y^{(k)})^{a_i}} \right)^{\frac{1}{a}} \qquad
\hbox{for $k \rightarrow \infty$.} 
\end{equation}
\end{theorem}

\begin{proof}
Since $\sum_{i = 0}^n \beta_i = a$, every solution $y$ to \eqref{eq:modifiedLP} satisfies
 $\sum_{i=0}^n y_i = 1$. Consider the map $\iota: (x_1, \ldots, x_n) \mapsto \frac{1}{|x|+1}(1, x_1, \ldots, x_n)$. 
  The map sending $b$ to $\beta$ is such that the diagram 
\begin{center}
\begin{tikzcd}
T_{A,c,\epsilon}^+\ar[r,hookrightarrow,"\iota"] \ar[d] & T_{{\cal A}, \gamma, \epsilon}^+ \cap \Delta_n \ar[d] \\
     {\rm pos}(A) \ar[r, "b \, \mapsto \beta"] & {\rm conv}(\mathcal{A})
\end{tikzcd}
\end{center}
is commutative. Here the vertical maps correspond to the isomorphism $T_{A,c,\epsilon}^+ \simeq {\rm pos}(A)$ in Theorem \ref{thm:homeolambda}. The diagram shows that \eqref{eq:modifiedLP} has 
the solution $y = y^*(\epsilon) = \iota(x^*(\epsilon))$. The iteration \eqref{eq:IPF} and its convergence can be derived from the proof of 
\cite[Theorem 1]{DR}.
\end{proof}

The geometric interpretation of Theorem \ref{thm:DR} is shown in Figure \ref{fig:DR}.
  The linear map given by $\mathcal{A}$ sends the 
probability simplex $\Delta_n$ onto
the polytope ${\rm conv}(\mathcal{A})$. Note that $z^{\gamma} $ lies in $ \Delta_n$.
The polytope $P_{{\cal A},\beta}$ is the set of all points in $\Delta_n$ that map to $\beta \in {\rm conv}(\mathcal{A})$ under $\mathcal{A}$. It is shown as a green triangle.
The toric variety $T_{{\cal A},\gamma,\epsilon} $ 
inside $\Delta_n$ is shown in blue, and ${\rm conv}(\mathcal{A})$ is the red line segment. The point $z^{\gamma} = y^{(0)}$ lies on 
$T_{{\cal A},\gamma,\epsilon} $ and is updated throughout the iteration. The solution $y = y^*(\epsilon) =\lim_{k \rightarrow \infty} y^{(k)}$ to \eqref{eq:modifiedLP} is the unique point in 
$T_{{\cal A},\gamma,\epsilon} \cap P_{{\cal A},\beta}$.

\begin{figure}[h]
\centering
\includegraphics[scale=0.35]{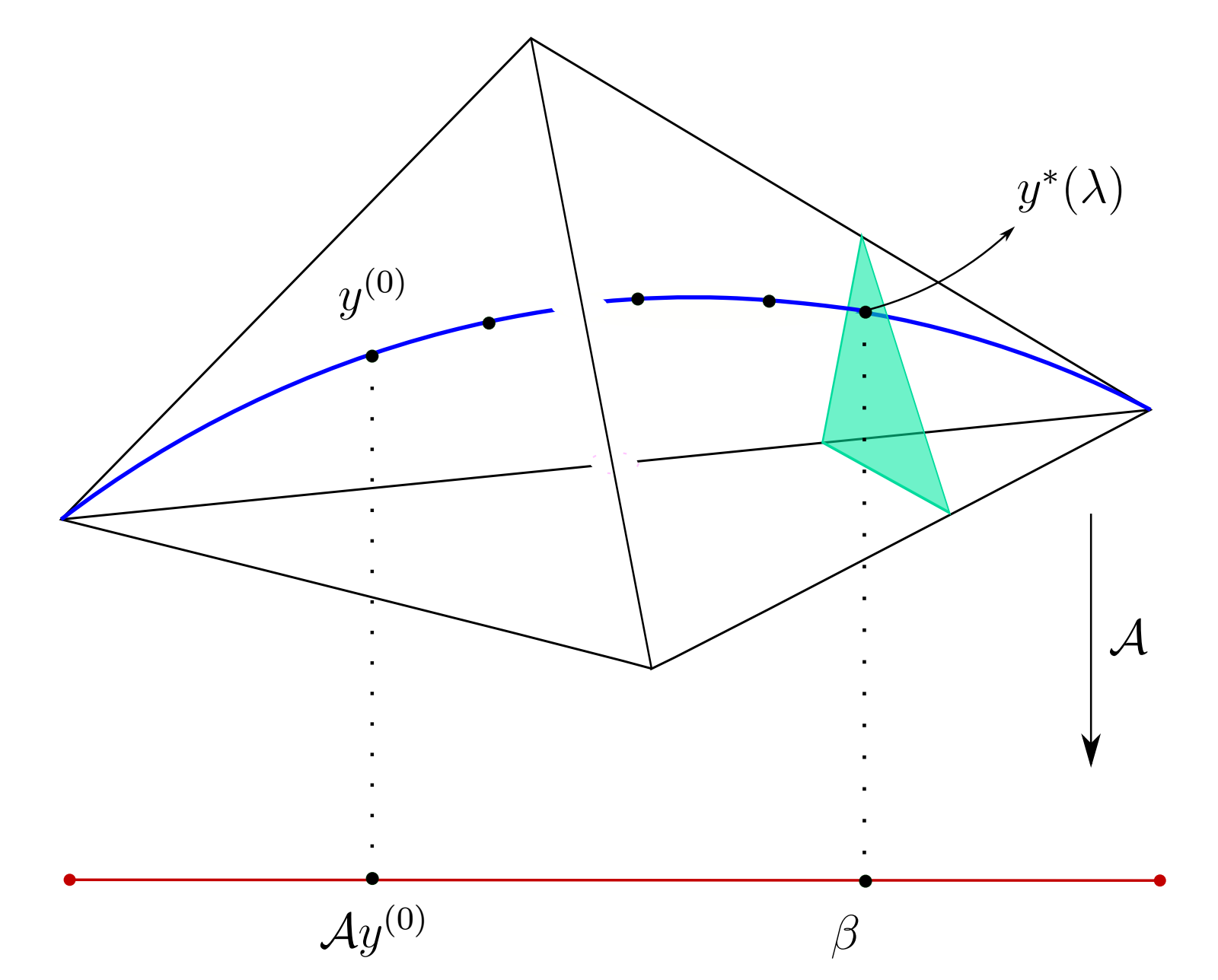}                                   
\caption{Illustration of the GIS algorithm from Theorem \ref{thm:DR}.}
\label{fig:DR}
\end{figure}

\smallskip

We now turn to the second step of the entropic interior point method, which consists of tracking $x^*(\epsilon)$ to the optimal vertex $x^*(0)$ of $P_{A,b}$. We assume that $c$ is sufficiently generic, so that $x^*(0)$ is indeed a vertex. Observe that, for all $\mu \in (0, \epsilon]$, we have $x^*(\mu) > 0$, $A \, x^*(\mu) = b$ and $x^*(\mu) \in T_{A,c,\mu}$. Equivalently, $x^*(\mu) = (t(\mu)^{a_j})_{j = 1, \ldots, n}$, where $t(\mu) \in \RR_{>0}^d$ is such that 
\begin{equation}  \label{eq:homotopystep2}
\sum_{j=1}^n a_{ij} \, \exp(-c_j/\mu) \,  t(\mu)^{a_j} \,=\, b_i \quad \text{for } \,i = 1, \ldots, d.
\end{equation}
The resulting functions $t(\mu)^{a_j}$ parametrize the entropic path
for $ \mu \in (0, \epsilon]$.
The starting 
point $t(\epsilon)$ is found by solving the 
binomial equations $x^*(\epsilon)_j = t(\epsilon)^{a_j} ,\, j = 1, \ldots n$. This can be done by a Smith normal form computation. The tracking for $\mu \rightarrow 0^+$ is carried out with standard \emph{predictor-corrector} techniques from numerical homotopy continuation \cite[Section~2.3]{sommese2005numerical}. 

We conclude with a toric interpretation of the homotopy \eqref{eq:homotopystep2}. For $\mu = \epsilon >0$, each of the Laurent polynomials in \eqref{eq:homotopystep2} defines a hypersurface in the projective toric variety $Y_P$ associated to the polytope $P = {\rm conv}(A \cup 0)$. There are ${\rm vol}(P)$ 
many solutions to \eqref{eq:homotopystep2} in $Y_P$, one of which gives $x^*(\epsilon)$. For $\mu \rightarrow 0$, this positive solution drifts to a lower dimensional torus orbit in $Y_P$, indicating which inequalities in $x^*(0) \geq 0$ are active. Identifying this orbit can be done by tracking the homotopy path $t(\mu)$ in homogeneous coordinates on $Y_P$.

\bigskip

\noindent
\footnotesize
{\bf Authors' addresses:}

\smallskip

\noindent Bernd Sturmfels,
MPI-MiS Leipzig and UC Berkeley
\hfill {\tt bernd@mis.mpg.de}

\noindent Simon Telen, 
 MPI-MiS Leipzig and CWI Amsterdam (current)
\hfill {\tt simon.telen@mis.mpg.de}

\noindent  Fran\c{c}ois-Xavier Vialard, LIGM,
Universit\'e Gustave Eiffel and INRIA Paris

\hfill {\tt francois-xavier.vialard@univ-eiffel.fr}

\noindent 
Max von~Renesse, Universit\"at Leipzig
\hfill {\tt renesse@uni-leipzig.de}


\begin{thebibliography}{10}
\begin{small}
\setlength{\itemsep}{-0.6mm}

\bibitem{Alla} X.~Allamigeon,
S.~Gaubert, A.~Aznag and Y.~Hamdi:
{\em The tropicalization of the entropic barrier},
{\tt arXiv:2010.10205}.

\bibitem{ABGJ} X.~Allamigeon, P.~Benchimol, S.~Gaubert
and M.~Joswig: {\em Log-barrier interior point methods are not strongly polynomial},
 SIAM J.~Appl.~Algebra Geom. {\bf 2} (2018) 140--178. 

 \bibitem{AKRS}
 C.~Am\'endola, K.~Kohn, P.~Reichenbach and A.~Seigal:
 {\em Toric invariant theory for maximum likelihood estimation in log-linear models},
 Algebraic Statistics {\bf 12} (2021) 187--211.



 
 \bibitem{CPSV} L.~Chizat, G.~Peyr\'e, B.~Schmitzer and F.-X.~Vialard:
{\em Scaling algorithms for unbalanced optimal transport problems},
Mathematics of Computation {\bf 87} (2018) 2563--2609.

\bibitem{CPSV2} L.~Chizat, G.~Peyr\'e, B.~Schmitzer and F.-X.~Vialard:
{\em Unbalanced optimal transport: dynamic and Kantorovich formulations},
Journal of Functional Analysis {\bf 274} (2018) 3090--3123.

\bibitem{Cuturi} M.~Cuturi: {\em
Sinkhorn distances: lightspeed computation of optimal transport},
Advances in Neural Information Processing Systems {\bf 26} (NIPS 2013).

\bibitem{DR}
J.~Darroch and D.~Ratcliff:
{\em Generalized iterative scaling for log-linear models},
Ann. Math. Statist. {\bf 43} (1972) 1470--1480.

 \bibitem{DRS}  
 J.~De Loera,   J.~Rambau and F.~Santos:
{\em Triangulations: Structures for Algorithms and Applications},
Algorithms and Computation in Mathematics, {\bf 25}, Springer, Berlin, 2010. 

\bibitem{DSV}
J.~De Loera, B.~Sturmfels and C.~Vinzant: {\em The central curve in linear programming},
   Foundations of Computational Mathematics  {\bf 12} (2012) 509--540. 
   
 \bibitem{DM} Y.~Dolinsky and H.~Mete Soner:
{\em Martingale optimal transport and robust hedging in continuous time},
Probab.~Theory Related Fields {\bf 160} (2014) 391--427.


 \bibitem{GO} G.~Guo and J.~Ob\l\'{o}j:
{\em Computational methods for martingale optimal transport problems},
Ann.~Appl. Probab. {\bf 29} (2019) 3311--3347.


\bibitem{KR}
J.~Karlsson and A.~Ringh: {\em Sinkhorn iterations for regularizing inverse problems
using optimal mass transport}, SIAM J.~Imaging Sciences {\bf 10} (2017) 1935--1962.

\bibitem{MR3763404}
M.~Liero, A.~Mielke  and G.~Savar\'{e}:
{\em Optimal entropy-transport problems and a new
  {H}ellinger-{K}antorovich distance between positive measures},
Invent. Math.~{\bf 211} (2018) 969--1117.


\bibitem{LuoTseng}  Z.~Luo and P.~Tseng:
{\em On the convergence of the coordinate descent method for convex differentiable minimization},
Journal of Optimization Theory and Applications {\bf 72} (1992) 7--35.

\bibitem{ASCB} L.~Pachter and B.~Sturmfels:
{\em Algebraic Statistics for Computational Biology},
Cambridge University Press, 2005.


\bibitem{she2018iterative} Y.~She and S.~Tang:
{\em Iterative proportional scaling revisited: a modern optimization perspective},
Journal of Computational and Graphical Statistics {\bf 28} (2019) 48--60.

\bibitem{sommese2005numerical}
A.~Sommese and C.~Wampler:
{\em The Numerical Solution of Systems of Polynomials Arising in
  Engineering and Science},
World Scientific Publishing, Hackensack, 2005.


\bibitem{GBCP} B.~Sturmfels: {\em Gr\"obner Bases and Convex Polytopes},
American Mathematical Society, Univ. Lectures Series, No 8, Providence, Rhode Island, 1996. 

\bibitem{StuTel} B.~Sturmfels and S.~Telen:
{\em Likelihood equations and scattering amplitudes},
Algebraic Statistics {\bf 12} (2021) 167--186.
 

\bibitem{Weed} J.~Weed: {\em An explicit analysis of the entropic penalty in
linear programming}, 31st Annual Conf.~on Learning Theory,
Proceedings~of Machine Learning Research
{\bf 75} (2018) 1--15.


 \end{small}
\end{thebibliography}
\end{document}